\documentclass[11pt,reqno]{article}
\usepackage{amsfonts}
\usepackage{amsmath}
\usepackage{amssymb}
\usepackage{epsfig}
\usepackage{amsthm}
\usepackage[english]{babel}
\usepackage[hypertex]{hyperref}
\newcommand{\N}{\mathbb N}

\newcommand{\E}{\mathbb E}

\newcommand{\hier}[3][X]{#1^{(#2)}_{#3}}

\newcommand{\half}{\frac{1}{2}}

\begin{document}
\theoremstyle{plain}
\newtheorem{thm}{Theorem}
\newtheorem{lem}[thm]{Lemma}
\newtheorem{cor}[thm]{Corollary}
\newtheorem{prop}[thm]{Proposition}

\title{Random walk attachment graphs}
\author{Chris Cannings and Jonathan Jordan\\
 University of Sheffield}
\maketitle
\begin{abstract}We consider the random walk attachment graph introduced by Saram\"{a}ki and Kaski and proposed as a mechanism to explain how behaviour similar to preferential attachment may appear requiring only local knowledge.  We show that if the length of the random walk is fixed then the resulting graphs can have properties significantly different from those of preferential attachment graphs, and in particular that in the case where the random walks are of length $1$ and each new vertex attaches to a single existing vertex the proportion of vertices which have degree $1$ tends to $1$, in contrast to preferential attachment models.\\ AMS 2010 Subject Classification: Primary 05C82. \\ Key words and phrases:random graphs; preferential
attachment; random walk.\end{abstract}

\section{Introduction}

There is currently great interest in the preferential attachment
model of network growth, usually called the Barab\'{a}si-Albert
\cite{ba,scalefreedefs} model, though it dates back at least to Yule \cite{yule}, and was
discussed also by Simon \cite{simon}. In the simplest version of this an
existing graph is incremented at each stage by adding a
single new vertex which then attaches to a single pre-existing
vertex; this latter is chosen from amongst those of the
pre-existing graph with probability proportional to the degree of
that vertex. In the Barab\'{a}si-Albert model the new vertex will
connect to $m$ vertices, where $m$ is fixed and is a parameter of the model, but here we only consider the
case $m=1$.  One of the best known properties of the model is that
it produces a power law degree distribution, as shown rigorously by
Bollob\'{a}s et al \cite{brst}.

One weakness of this model and its
generalisations is that this implicitly requires a calculation
across all the existing vertices, or at least a knowledge of the
total degree (sum of the vertex degrees) of the graph. This
requirement then destroys the potential for this model
to have emergent properties from local behaviour.

A possible solution to this was proposed by Saram\"{a}ki and Kaski
\cite{sarkaski}. In their model the new vertex simply chooses a single vertex
from the graph and then executes a random walk of length $\ell$ step
initiated from that vertex. Saram\"{a}ki and Kaski \cite{sarkaski} and Evans
and Saram\"{a}ki \cite{evans} claim that this reproduces the
Barab\'{a}si-Albert degree distribution, even when $\ell=1$.  It is
clear that this is the    case if the random walk is run for long
enough to have converged to its stationary distribution. However
we will prove that in the particular case $\ell=1$ the degree sequence does not converge to a power law distribution, but rather to a degenerate limiting distribution in which almost every
vertex has degree $1$.

\section{The Model}

Let $G_0$ be an arbitrary (perhaps connected) graph, with $v_0$ vertices and $e_0$ edges. Form $G_{n+1}$
from $G_n$ by adding a single vertex.  This vertex chooses a single
vertex (i.e. this corresponds to $m=1$ in the Barab\'{a}si-Albert model) to connect to by picking a vertex uniformly at
random in $G_n$ and then, conditional on the vertex chosen, performing a simple
random walk of length $\ell$ on $G_n$, starting from the randomly chosen vertex, and then choosing to connect to the destination vertex.  Most of the time we will assume that $\ell$ is deterministic, but we will also consider a particular case where $\ell$ is replaced by a random variable.

\section{Number of leaves}

We first consider the number of leaves in the graph.  Let $\hier[p]{n}{d}$ be the proportion of vertices in $G_n$ with
degree $d$, and let $L_n=\hier[p]{n}{1}$, i.e. the proportion of
leaves.  The number of edges in $G_n$ will be $n+e_0$, the total
degree will thus be $2(n+e_0)$, and the number of vertices will be
$n+v_0$.  Let $V_n$ be
the vertex initially chosen at random at step $n$, and let $W_n$ be
the vertex selected by the random walk, so the new vertex connects
to $W_n$. We now prove the main result, which applies to the case where $\ell=1$.

\begin{thm}\label{l1}When $\ell=1$, as $n\to\infty$, $L_n\to 1$, almost surely.\end{thm}

\begin{proof}We assume that $G_0$ is not a star.  If $G_0$ is a star,
then it is clear that, with probability $1$, $G_n$ will eventually
not be a star, so we can just wait until this happens and re-label
the first non-star graph as $G_0$.  If $G_n$ is not a star each vertex has at least one neighbour which is not a leaf,
and in particular no leaves have a leaf as their neighbour. If $V_n$ is a leaf,
which has probability $L_n$, then $W_n$ will be one of its
neighbours, which will not be a leaf, so in this case the number of
leaves increases by $1$.  Hence, considering the conditional expectation of the number of leaves in $G_{n+1}$,
\begin{equation}\label{submart}\E((n+v_0+1)L_{n+1}|G_n)\geq
(n+v_0)L_n+L_n=(n+v_0+1)L_n,\end{equation} and so $\E(L_{n+1}|G_n)\geq L_n$ and
so $(L_n)_{n \in \N}$ is a submartingale taking values in $[0,1]$,
and thus converges almost surely and in $\mathcal{L}^2$ to a limit, which we call
$L_{\infty}$.

To show that $L_{\infty}=1$ almost surely, note that conditional on $V_n$ having
degree $d$ the probability of $W_n$ not being a leaf is at
least $1/d$, so we can make \eqref{submart} sharper, getting
\begin{equation}\E(L_{n+1}|G_n) \geq
L_n+\sum_{d=2}^{\infty}\frac{\hier[p]{n}{d}}{(n+v_0+1)d}.\end{equation}

The total degree of non-leaves in $G_n$ is $2(n+e_0)-L_n(n+v_0)=(2-L_n)(n+v_0)+2(e_0-v_0)$, and the
number of non-leaves is $(1-L_n)(n+v_0)$, so the average degree of
non-leaves is $\frac{2-L_n}{1-L_n}+\frac{2(e_0-v_0)}{(n+v_0)(1-L_n)}$.  Hence at least half the non-leaves have degree at most
$2\left(\frac{2-L_n}{1-L_n}+\frac{2(e_0-v_0)}{(n+v_0)(1-L_n)}\right)$ and so
\begin{equation}\E(L_{n+1}|G_n) \geq
L_n+\frac{1-L_n}{2(n+1)}\left(2\left(\frac{2-L_n}{1-L_n}+\frac{2(e_0-v_0)}{(n+v_0)(1-L_n)}\right)\right)^{-1}\end{equation} and so \begin{equation}\label{expiter}\E(L_{n+1}) \geq
\E(L_n)+\frac{1}{2(n+1)}\E\left(\frac{1-L_n}{2}\left(\frac{2-L_n}{1-L_n}+\frac{2(e_0-v_0)}{(n+v_0)(1-L_n)}\right)^{-1}\right).\end{equation}

If $\E(L_{\infty})=\lim_{n\to\infty}\E(L_n)<1$, then for some fixed $c<1$ we must have $L_{n}\leq c$ with positive probability.  The expectation on the right of \eqref{expiter} is then bounded away from zero for large $n$, giving a contradiction and showing that $\E(L_{\infty})=1$ and thus that
 $L_{\infty}=1$ almost surely.\end{proof}

It should be noted that the argument for Theorem \ref{l1} is dependent on the walk length being fixed at $1$.  For example, define a sequence of random variables $(X_n)_{n\in\N}$ which are independent and identically distributed with $P(X_n=0)=p$ and $P(X_n=1)=1-p$, and let the walk length from $V_n$ to $W_n$ be $X_n$, rather than a fixed $\ell$ as previously.

Then, by the same argument as before
$$\E(L_{n+1}-L_n|G_n,X_{n+1}=1)\geq \frac{1-L_n}{2}\frac{1-L_n}{2(n+v_0+1)(2-L_n)}+O(n^{-2}).$$  As there can be at
most one more leaf in $G_{n+1}$ than in $G_n$, we also have  $$\E(L_{n+1}-L_n|G_n,X_{n+1}=1)\leq
\frac{1-L_n}{n+v_0+1}+O(n^{-2}).$$

Also, if there are no random walk steps from the initially chosen vertex the probability that the new vertex connects to
a leaf is simply $L_n$, so $$\E((n+v_0+1)L_{n+1}|G_n,X_{n+1}=0)=(n+v_0)L_n+1-L_n,$$ and hence
$$\E(L_{n+1}-L_n|G_n,X_{n+1}=0)=\frac{1}{n+v_0+1}(1-2L_n).$$

So, if we have $X_{n}=0$ with probability $p$ and $1$ with probability $1-p$ for all $n$ independently of each other \begin{equation}\label{ub}\E(L_{n+1}-L_n|G_n)\geq
\frac{1}{n+v_0+1}\left[p(1-2\lambda)+(1-p)\frac{(1-\lambda)^2}{4(2-\lambda)}\right]+O(n^{-2}).\end{equation}  Similarly, \begin{equation}\label{lb}\E(L_{n+1}-L_n|G_n)\leq
\frac{1}{n+v_0+1}\left[1-\lambda(1+p)\right]+O(n^{-2}).\end{equation}

The right hand side of \eqref{ub} is negative if $$L_n<\frac{1+9p-2\sqrt{8p^2+p}}{1+7p}$$ and $n$ is sufficiently large and the right hand side of
\eqref{lb} is negative if $L_n>\frac{1}{1+p}$ and $n$ is sufficiently large.

Note that $$\frac{1+9p-2\sqrt{8p^2+p}}{1+7p} -\frac{1}{1+p} \ge 0$$ for $p
\in [0,1]$ with equality only at $p=0$ and $p=1$, and that $$\frac{1+9p-2\sqrt{8p^2+p}}{1+7p}\leq 1,$$ with equality only if $p=0$.

A version of the argument of Lemma 2.6 of \cite{pemantlesurvey} now shows that, almost surely, $$\liminf_{n\to\infty}L_n\geq \frac{1}{1+p}$$ and $$\limsup_{n\to\infty}L_n\leq \frac{1+9p-2\sqrt{8p^2+p}}{1+7p}.$$  So we do not get a similar result to Theorem \ref{l1} in this setting.

\section{$G_0$ Bipartite}

We now consider a special case which demonstrates that, for all odd
$\ell$, the random walk model of \cite{sarkaski} differs
fundamentally from that of the Barab\'{a}si-Albert model.

Assume that $G_0$ is a bipartite graph, with the two parts coloured
as red and blue. Then, in both models, for all $n$ the graph $G_n$ will be
bipartite, and the parts can be coloured red and blue consistently
for each $n$. Let the proportion of red vertices in $G_n$ be $R_n$.
We begin with the random walk model.

\begin{thm}
We have $R_{\infty}$ such that $R_n$ converges almost surely to
$R_{\infty}$. If $\ell$ is even, then $R_{\infty}=\half$, almost
surely, while if $\ell$ is odd $R_{\infty}$ is a random variable
with a Beta distribution.
\end{thm}
\begin{proof}
Conditional on $G_n$, $V_n$ will be red with probability $R_n$. If
$\ell$ is odd $W_n$ will be of opposite colour to $V_n$, which
implies that the new vertex (which connects to $W_n$) will be of the
same colour as $V_n$, and thus, conditional on $G_n$, will be red
with probability $R_n$ and blue with probability $1-R_n$.  Hence in
this case the colours of vertices are equivalent to the colours of
the balls in a standard P\'{o}lya urn (where when a ball is drawn
two of the same colour are returned), and so by classical results on the P\'{o}lya urn (see, for example, Theorem 2.1 in \cite{pemantlesurvey}) $R_n$ converges almost
surely to $R_{\infty}$ where $R_{\infty}$ has a Beta distribution
whose parameters depend on $G_0$.

If $\ell$ is even then $W_n$ is of the same colour as $V_n$ and so
the new vertex is of opposite colour to $V_n$.  Hence this case
corresponds to a two-colour generalised P\'{o}lya urn where a ball
is selected and a ball of the opposite colour is added, namely a
Friedman urn with $\alpha=0$ and $\beta=1$.  In this case
$R_n\to\half$ almost surely; see for example Freedman \cite{freedman}, and Theorem 2.2 in \cite{pemantlesurvey}.
\end{proof}

\begin{thm} In the Barab\'{a}si-Albert model $R_{\infty}=\half$
almost surely.
\end{thm}
\begin{proof}
In this model it is possible to associate the selection of a vertex
with an urn model by considering half-edges, and giving each
half-edge the colour of its associated vertex, i.e. each edge is
split into a red half and a blue half. The selection of a vertex
with probability proportional to its degree is then equivalent to
selecting a half-edge uniformly at random and then selecting the
associated vertex.  As the new edge added in $G_{n+1}$ will always
consist of a blue half and a red half, the proportion of red
half-edges must converge to $\half$, and as a red vertex is added if
and only if a blue vertex is selected, the proportion of red
vertices will converge to $\half$, almost surely.
\end{proof}

So in this respect the behaviour of the random walk model is
different from the Barab\'{a}si-Albert model when $\ell$ is odd,
regardless of the size of $\ell$.

\section{Discussion}

We have demonstrated that the model of Saram\"{a}ki and Kaski is
fundamentally different from that of Barab\'{a}si and Albert, unless
we allow an indefinite length for the random walk component. It does
have the advantage of not requiring a global calculation, retaining
the local behaviour characteristic which is desirable in models of
emergent behaviour. An alternate approach might be to imagine that the
addition of edges is affected by the vertices in $G_n$, rather than
by the new vertex. Thus each vertex in $G_n$ could link to a new
vertex as it arises with probability proportional to its degree,
independently of all other vertices, as in the variant of preferential attachment studied by Dereich and M\"{o}rters \cite{dm1,dm2}. This, of course, destroys one of
the usual assumptions of the preferential attachment model that the
number of new links is some fixed value $m$, though we could
substitute the condition that the average number added was fixed.

The urn model approach is interesting particularly since there is
much known about these (see for example the survey paper by Pemantle \cite{pemantlesurvey}). We might
generalise the model to consider directed graphs where there are $k$
colours $c_i;~~ i=0,k-1$, with directed edges only between a vertex
of colour $c_i$ and one of colour $c_{(i+1)(\mathrm{mod}~k)}$. When a new
vertex is added it links at random to a vertex and then takes $\ell$
random steps along directed edges, its colour then being determined.
The case $\ell \ne 0(\mathrm{mod}~k)$ will have the proportions of each
colour converging to $1/k$, whereas for $\ell = 0(\mathrm{mod}~k)$ there will
be a Dirichlet distribution with parameters depending on $G_0$.

\section{Acknowledgement}

The first author acknowledges support from the European Union
through funding under FP7-ICT-2011-8 project HIERATIC (316705).

\end{document}